\documentclass[12pt]{amsart}

\newtheorem{theorem}{Theorem}[section]

\newtheorem{corollary}[theorem]{Corollary}

\theoremstyle{definition}

\newtheorem{remark}[theorem]{Remark}

\theoremstyle{parrafo}

\begin{document}

\title[]{Selfimprovemvent of the inequality between arithmetic and geometric means}

\author{J. M. Aldaz}
\address{PERMANENT ADDRESS: Departamento de Matem\'aticas y Computaci\'on,
Universidad  de La Rioja, 26004 Logro\~no, La Rioja, Spain.}
\email{jesus.munarrizaldaz@dmc.unirioja.es}

\address{CURRENT ADDRESS: Departamento de Matem\'aticas,
Universidad  Aut\'onoma de Madrid, Cantoblanco 28049, Madrid, Spain.}
\email{jesus.munarriz@uam.es}


\thanks{The author was partially supported by Grant MTM2006-13000-C03-03 of the
D.G.I. of Spain}









\maketitle


\markboth{J. M. Aldaz}{AM-GM}


It is well known that the AM-GM inequality
has selfimproving properties.  Let $x_i\ge 0$ for
$i=1,\dots, n$. The classical, equal weights case, states that
\begin{equation}\label{AMGc}
\prod_{i=1}^n x_i^{1/n}
\le
\frac{1}{n} \sum_{i=1}^n x_i.
\end{equation}
Let
$\alpha_i > 0$ satisfy $\sum_{i=1}^n \alpha_i = 1$.
Inequality (\ref{AMGc})  selfimproves  to the rational
weights case simply via repetition of terms, and to the case of real
 weights $\alpha_i$ just by taking limits. So the
general
AM-GM inequality
\begin{equation}\label{AMG}
\prod_{i=1}^n x_i^{\alpha_i}
\le
\sum_{i=1}^n \alpha_i x_i
\end{equation}
follows. There is a second way in which the AM-GM inequality selfimproves.
Let $s > 0$ and use the change of variables $x_i = y_i^s$. Substituting
in (\ref{AMG}) and taking $s$-th roots we get
\begin{equation}\label{AMGr}
\prod_{i=1}^n y_i^{\alpha_i}
\le
\left(\sum_{i=1}^n \alpha_i y_i^{s}\right)^{1/s}.
\end{equation}
Now for $0 < s < 1$, Jensen's inequality tells us that $\left(\sum_{i=1}^n \alpha_i y_i^{s}\right)^{1/s} \le \sum_{i=1}^n \alpha_i y_i$ since $t^s$ is concave, and furthermore the
inequality is strict unless $y_1=\dots = y_n$ (this follows from the equality case in Jensen's inequality). So (\ref{AMG}) automatically proves a
family of
better inequalities; it ``pulls itself by its bootstraps". The particular case $s=1/2$ immediately leads to a natural
and useful refinement of (\ref{AMG}).

\begin{theorem}\label{AMGMr}  For $i=1,\dots, n$, let $x_i\ge 0$,
 and let
$\alpha_i > 0$ satisfy $\sum_{i=1}^n \alpha_i = 1$. Then
\begin{equation}\label{refAMGM}
\prod_{i=1}^n x_i^{\alpha_i}
\le
\sum_{i=1}^n \alpha_i x_i -
\sum_{i=1}^n \alpha_i \left(x_i^{1/2}- \sum_{k=1}^n \alpha_k x_k^{1/2}\right)^2.
\end{equation}
\end{theorem}
Note that the right most term of (\ref{refAMGM}) is the variance
$\operatorname{Var}(x^{1/2})$ of
the vector $x^{1/2} = (x_1^{1/2},\dots,x_n^{1/2})$ with respect to the probability $\sum_{i=1}^n \alpha_i \delta_{x_i}$. So a large variance
(of $x^{1/2}$) pushes the arithmetic and geometric means apart.

\begin{proof} Recalling that
$\operatorname{Var}(X) = E(X^2) - (E(X))^2 = E([X - E(X)]^2) $, and
using (\ref{AMGr}) with $s = 1/2$, we obtain
\begin{equation*}
\sum_{i=1}^n \alpha_i x_i -
\prod_{i=1}^n x_i^{\alpha_i}
\ge
\sum_{i=1}^n \alpha_i x_i -
\left(\sum_{k=1}^n \alpha_k x_k^{1/2}\right)^{2}
=
\sum_{i=1}^n \alpha_i \left(x_i^{1/2}- \sum_{k=1}^n \alpha_k x_k^{1/2}\right)^2 = \operatorname{Var}(x^{1/2}).
\end{equation*}
\end{proof}

This refinement of the AM-GM inequality leads to an improvement of H\"older's inequality for several functions.

\begin{corollary}\label{betterhold}  For $i = 1,\dots, n$, let $1 < p_i < \infty$
be such that $p_1^{-1} + \cdots + p_n^{-1} = 1$, and let $0\le f_i\in L^{p_i}$
satisfy  $\|f_i\|_{p_i}  > 0$.  Then
\begin{equation}\label{bonhold}
\left\|\prod_{i=1}^n f_i\right\|_1
\le
\prod_{i=1}^n\|f_i\|_{p_i} \left(1 - \sum_{i=1}^n \frac1{p_i}
\left\|\frac{f_i^{p_i/2}}{\|f_i\|_{p_i}^{p_i/2}} -
\sum_{k=1}^n \frac1{p_k} \frac{f_k^{p_k/2}}{\|f_k\|_{p_k}^{p_k/2}}\right\|_2^2\right).
\end{equation}
\end{corollary}

\begin{proof}  Set $\alpha_i = p_i^{-1}$ and $x_i = f_i^{p_i}(u)/\|f_i\|_{p_i}^{p_i}$ in (\ref{refAMGM}). To obtain (\ref{bonhold}), integrate
and multiply both sides by $\prod_{i=1}^n\|f_i\|_{p_i}$.
\end{proof}

\begin{remark} Inequality (\ref{refAMGM}) was suggested
by the following result of D. I. Cartwright and M. J. Field
(cf. \cite{CaFi}; cf.  also \cite{Alz} and \cite{Me} for additional
refinements along these lines). Let $0 < m =\min\{x_1,\dots, x_n\}$ and let $M=\max\{x_1,\dots, x_n\}$. Then
\begin{equation}\label{cafi}
\frac{1}{2M}  \sum_{i=1}^n \alpha_i \left(x_i- \sum_{k=1}^n \alpha_k x_k\right)^2
\le
\sum_{i=1}^n \alpha_i x_i - \prod_{i=1}^n x_i^{\alpha_i}
\le
\frac{1}{2m}  \sum_{i=1}^n \alpha_i \left(x_i- \sum_{k=1}^n \alpha_k x_k\right)^2.
\end{equation}
The motivation to search for variants of (\ref{cafi}) comes the fact that
it is not well suited to the
particular application considered here (refining
 H\"older's inequality).  One
would need to assume that $|f_i|\le M$ almost everywhere.
We give bounds using the variance of $x^{1/2}$ instead of the variance of $x$ in order to ensure the integrability of the functions involved,
and also to obtain the same homogeneity  on both sides of (\ref{refAMGM}).
\end{remark}

\begin{remark} The difference between the arithmetic and geometric
means is in general not comparable to $\operatorname{Var} (x^{1/2})$.
To see this, it is enough to consider the equal weights case, with
$n >>1$, $x_1 = 0$, and $x_2=\dots = x_n = 1$. Or the case where
$n = 2$, and one of the weights is much larger than the other.
But perhaps it is possible to give an upper bound
for $\sum_{i=1}^n \alpha_i x_i - \prod_{i=1}^n x_i^{\alpha_i}$ using $\operatorname{Var} (x^{1/2})$
times some polynomial function of 1 over the smallest weight. This would lead
to the
same type of application as above. In fact,
 for the special case $n=2$
 a two sided, sharper version
of  (\ref{refAMGM}) appears in Lemma 2.1 of \cite{Al}. It is
not clear to me how to extend this sharper version to $n > 2$.
\end{remark}

\begin{remark} When $n=2$, inequality (\ref{bonhold}) reduces to
\begin{equation}\label{bonhold2}
\|fg\|_1
 \le  \|f\|_p\|g\|_q \left(1 - \frac1{pq}
\left\|\frac{f^{p/2}}{\|f\|_p^{p/2}}-\frac{g^{q/2}}{\|g\|_q^{q/2}}\right\|_2^2\right),
\end{equation}
where $p$ and $q$ are conjugate exponents, $0\le f\in L^p $, $0\le g\in L^q$, $\|f\|_p >0$, and  $\|g\|_q > 0$.
In addition to providing a lower bound, with
$1/ \min\{p,q\}$ instead of $1/ (pq)$,
 Lemma 2.1 of \cite{Al}
yields a slightly better upper bound: $1/(pq) = 1/(p + q)$  can be replaced by $1/\max\{p,q\}$. But we note that
 (\ref{bonhold2}) suffices, via the standard argument, to give
a refinement of the triangle inequality for $L^p$ spaces, $1 < p <\infty$,
 which in
 turn leads to a fairly straightforward proof of uniform convexity
in the real valued case (arguing as in  \cite{Al}). So the selfimproving properties of the AM-GM inequality have repercussions beyond what one might expect.
\end{remark}

\begin{remark} Note that $f_i^{p_i/2}/\|f_i\|_{p_i}^{p_i/2}$ is just
a unit vector in $L^2$.
The strategy underlying inequality (\ref{bonhold}) is to normalize all
functions and map them into $L^2$, which becomes the common measuring
ground where dispersion around the mean is determined. When $n = 2$,
the correction term reduces to a function of the angular distance between
$f^{p/2}$ and $g^{q/2}$.
\end{remark}

\end{document}